\documentclass[11pt]{article}%

\usepackage[english]{babel}
\usepackage{amssymb}
\usepackage{amsmath,enumerate,rotate}
\usepackage{amssymb,latexsym}
\usepackage{epsfig}
\usepackage{graphicx, graphics,float,subfigure}
\usepackage{amsthm}
\usepackage{natbib}

\usepackage[colorlinks=true,linkcolor=blue]{hyperref}%
\usepackage{graphicx}
\hypersetup{urlcolor=blue, citecolor=blue, colorlinks=true, linkcolor=blue}
\usepackage[margin=1in]{geometry}
\usepackage{verbatim}

\usepackage{soul}
\usepackage{color}
\usepackage[mathcal]{eucal}
\usepackage{mathrsfs} 
\usepackage{verbatim}
\usepackage{url}
\usepackage{setspace}
\usepackage{mathtools} 

\theoremstyle{plain}
\newtheorem{theorem}{Theorem}
\newtheorem*{theorem*}{Theorem}

\newtheorem*{corollary*}{Corollary}
\newtheorem*{definition*}{Definition}

\theoremstyle{remark}

\def\R{\mathbb{R}}

\newcommand\norm[1]{{\left\Vert{#1}\right\Vert}}

\def\ss{\boldsymbol{s}}

\makeatletter
\def\varphiall{\@ifnextchar[{\varphiall@i}{\varphiall@i[]}}
\def\varphiall@i[#1]{\@ifnextchar[{\varphiall@ii{#1}}{\varphiall@ii{#1}[#1]}}
\def\varphiall@ii#1[#2]{\varphi_{\mu,\kappa,\beta_{#1},\sigma^2_{#2}}}
\def\hatvarphiall{\@ifnextchar[{\hatvarphiall@i}{\hatvarphiall@i[]}}
\def\hatvarphiall@i[#1]{\@ifnextchar[{\hatvarphiall@ii{#1}}{\hatvarphiall@ii{#1}[#1]}}
\def\hatvarphiall@ii#1[#2]{\widehat{\varphi}_{\mu,\kappa,\beta_{#1},\sigma^2_{#2}}}
\makeatother

\def\R{\mathbb{R}}

\def\d{\textrm{d}}

\def\ve{\varepsilon}

\def\le{\left(}
\def\ri{\right)}

\begin{document}
	

\fontsize{12}{14pt plus.8pt minus .6pt}\selectfont \vspace{0.8pc} \vspace{4pt}
\centerline{\LARGE \bf Convergence Arguments to Bridge   }
\vspace{.4cm}
 \centerline{\LARGE \bf  Cauchy and Mat{\'e}rn Covariance Functions}  \vspace{8pt}
\begin{center}
	{\large {\sc Tarik Faouzi\footnote{\baselineskip=10pt Departamento de Matem\'atica y Ciencia de la Computaci\'on, Universidad de la USACH, Santiago, Chile.}, Emilio Porcu\footnote{\baselineskip=10pt Department of Mathematics, Khalifa University,
	Abu Dhabi, $\&$ School of Computer Science and Statistics, Trinity College Dublin.},
	{\large {\sc Igor Kondrashuk}}\footnote{\baselineskip=10pt
       Grupo de Matem\'{a}tica {A}plicada  {\rm \&}  Centro de Ciencias Exactas  {\rm \&}
       Departmento de Ciencias B\'{a}sicas, Universidad del B\'{i}o-B\'{i}o, Campus Fernando May, Av. Andres Bello 720, Casilla 447, Chill\'{a}n, Chile}
       and Moreno Bevilacqua\footnote{\baselineskip=10pt
       Department of Statistics, Universidad Adolfo Iba{\~n}ez, Vi{\~na} del Mar, Chile.
	}}}, \\

\vspace{1cm}

\begin{abstract}
The Mat{\'e}rn and the Generalized Cauchy families of covariance functions
have a prominent role in spatial statistics as well as in a wealth of statistical applications. The Mat{\'e}rn family is crucial to index mean-square differentiability of the associated Gaussian random field; the Cauchy family is a decoupler of the fractal dimension and Hurst effect for Gaussian random fields that are not self-similar. \\
Our effort is devoted to prove that a scale-dependent family of covariance functions, obtained as a reparameterization of the Generalized Cauchy family, converges to a particular case of the Mat{\'e}rn family, providing a somewhat surprising bridge between covariance models with light tails and covariance models that allow for long memory effect.

\vspace{0.8cm}


\end{abstract}

\newpage

\end{center}

\section{Introduction}

Covariance functions are fundamental to spatial statistics \citep{Stein:1999}. The use of covariance functions for modeling interpolation and prediction has been largely influential within the spatial
statistics and machine learning communities, with a wealth of applications to many applied disciplines. For comprehensive reviews of the subject, the reader is referred to  \cite{porcu2018modeling}  and \cite{porcu202130}. \\
The Mat{\'e}rn family of covariance functions has been the cornerstone in spatial statistics for over $30$ years. Its importance stems from the inclusion of a parameter that allows to control the fractal dimension
and the mean square differentiability of the associated Gaussian process. This is in turn a fundamental aspect in evaluating predictive performance of covariance functions under infill asymptotics \citep{zhang}.
The nice analytical form of the spectral density associated with the Mat{\'e}rn covariance function allows for  theoretical analysis of the properties of maximum likelihood estimators \citep{zhang}, approximate likelihood \citep{FGN,bevb:12, kaufman}
and misspecified linear unbiased prediction \citep{Stein:1999} under infill asymptotics. A wealth of results is also available within SPDE's with Gaussian Markov approximations \citep{Lindgren} as well as in the numerical analysis literature.
We refer the reader to \cite{SSS} for more details. An historical treatise about the Mat{\'e}rn covariance function is
provided by \cite{guttorp2006studies}. \\
The generalized Cauchy family of covariance functions \citep{gneiting2004stochastic} has been largely used in the statistics and engineering literature as it allows for decoupling the fractal dimension and the Hurst effect
of the associated Gaussian field. So far, the only alternative to the generalized Cauchy family is represented by the Dagum family \citep{berg2008dagum}, having similar properties in terms of decoupling. A thorough comparison between the two models in terms of properties of their spectral densities can be found in \cite{faouzi2021deep}. For applications related to turbulence, the reader is referred to \cite{laudani2021streamwise}, while the importance of both generalized Cauchy and Dagum models is discussed in \cite{zhang2022elastodynamic} and in \cite{nishawala2020random}.

The fundamental message is that different parametric families of covariance functions allow for different properties of the associated Gaussian random field. For many decades, such families - and consequently,
their properties, have been treated separately. Recently, there has been a major effort to relate different families of covariance functions. \cite{bevilacqua2022unifying} proved that the Mat{\'e}rn family
is a special case of the generalized Wendland family of compactly supported correlation functions. Specifically,  they proved that the reparameterized version of the generalized Wendland family generalizes the Matérn model
which is attained as a special limit case. This implies that the reparametrized generalized Wendland model is more flexible than the Matérn model, with an extra-parameter that allows for switching from compactly
to globally supported covariance functions. A similar approach is adopted by \cite{emery2021gauss} for another class of compactly supported covariance functions.

The present paper provides some effort to relate the Generalized Cauchy with the Mat{\'e}rn covariance function. To do so, we provide a scale-dependent reparameterization of the generalized Cauchy by replacing
the scaling factor with a function depending on other parameters. Such an {\em escamotage} is then proved successful as it opens for a convergence-type result that depicts a specific case of the Mat{\'e}rn model
as a special case of the Generalized Cauchy covariance.

From a technical viewpoint, the mentioned convergence is actually supported by (a) pointwise convergence of the spectral densities, and (b) uniform convergence of the covariance functions. The proofs are based on Mellin-Barnes transforms, being less familiar to the statistical community and more common in harmonic analysis and theoretical physics, and the reader is referred to \cite{Kondrashuk:2008ec} for a thorough account.

The plan of the paper is the following. Section \ref{sec2} provides the necessary backround, which is split into three parts. The first introduces to the environment of positive definite radial functions. The second describes the two parametric families of covariance functions used in this paper. The third part is useful for those who might be interested in the technical aspects of this papers, especially on contour-loop integration and Mellin-Barnes transforms. The combination of the three novel theoretical results as in Section \ref{sec3} provides the main message of this paper. We conclude our findings with a short discussion.

\section{Background Material}\label{sec2}

\subsection{Positive Definite Radial Functions}

We denote by $\{Z(\ss), \ss \in \R^d \} $ a centered Gaussian random field in $\R^d$, with a continuous stationary covariance function $C:\R^d
\to \R$. We further assume isotropy, that is there exists a continuous mappings
$\phi:[0,\infty) \to \R$ with $\phi(0) =1$, such that
\begin{equation} \label{generator}
{\rm cov} \left ( Z(\ss), Z(\ss^{\prime}) \right )= C(\ss^{\prime}-\ss)=  \sigma^2 \phi(\|\ss^{\prime } -\ss \|),
\end{equation}
with $\ss,\ss^{\prime} \in \R^d$, and $\|\cdot\|$ denoting the Euclidean norm. Here, $\sigma^2$ is the varince of $Z$.  The function $C$ is called \emph{isotropic} or \emph{radially symmetric}, and the function $\phi$ its \emph{radial part}.
\cite{Shoe38} characterized the class for continuous functions $\phi$ for which (\ref{generator}) is true for some random fields $Z$ in $\R^d$. The function $\phi$ admits a uniquely determined scale mixture representation of the type
$$ \phi(r)= \int_{0}^{\infty} \Omega_{d}(r \xi) F(\d \xi), \qquad r \ge 0,$$
with $\Omega_{d}(r)= r^{-(d-2)/2}J_{(d-2)/2}(r)$ and $J_{\nu}$ a Bessel function of order $\nu$. Here, $F$ is a probability measure. The function $\phi$ is the uniquely determined characteristic function of a random vector, $\mathbf X$, such that $\mathbf X = \mathbf U \cdot R$, where equality is intended in the distributional sense, where $\mathbf U$ is uniformly distributed over the spherical shell of $\R^d$, $R$ is a nonnegative random variable with probability distribution, $F$, and where $\mathbf U$ and $R$ are independent \citep{daley-porcu}.  The derivative of $F$ is called the \emph{isotropic spectral density}. If $\phi$ is absolutely integrable, then the Fourier inversion (the Hankel transform) becomes possible. The Fourier transforms of radial parts of positive definite functions in $\R^d$, for a given $d$, have a simple expression, as reported in \cite{Stein:1999} and \cite{Yaglom:1987}. Specifically, we define the isotropic spectral density in $\R^d$ as
\[
 \widehat{\phi}(z)= \frac{z^{1-d/2}}{(2 \pi)^{d/2}} \int_{0}^{\infty} u^{d/2} J_{d/2-1}(uz)  \phi(u) \,{\rm d} u, \qquad z \ge 0,
\]
where $z = \|\boldsymbol{z}\|$ for $\boldsymbol{z} \in \R^d$.

\subsection{Two Parametric Families of Isotropic Covariance Functions}
We are discussing below two parametric families of functions $\phi$ associated with Equation (\ref{generator}). \\
The Mat\'ern family is defined through
\begin{equation}
\label{Matern}
{\cal M}_{\nu,\alpha}(r)=
 \frac{2^{1-\nu}}{\Gamma(\nu)} \left (\frac{r}{\alpha}
  \right )^{\nu} {\cal K}_{\nu} \left (\frac{r}{\alpha} \right ),
  \qquad r \ge 0,
\end{equation}
for any positive values of the scaling parameter, $\alpha$, and the smoothing parameter, $\nu$.
Here,
${\cal K}_{\nu}$ is a modified Bessel function of the second kind of
order $\nu$. The parameter $\nu$ is crucial for the
differentiability at the origin and, as a consequence, for the
degree of  differentiability of the associated sample paths.
Specifically,  for a
positive integer $k$, the sample paths are $k$ times differentiable if
and only if $\nu>k$.

The other parametric family we are going to discuss is termed Generalized Cauchy after \cite{gneiting2004stochastic}.
The function ${\cal C}_{\delta,\lambda,\gamma} : [0,\infty) \to \R$ is defined as
\begin{equation} \label{averrohe}
{\cal C}_{\delta,\lambda,\gamma}(r)=\frac{1
}{\Big (1+\Big( \frac{r}{\gamma} \Big )^\delta \Big )^{\lambda}},
  \qquad r \ge 0.
\end{equation}
The function ${\cal C}_{\delta,\lambda,\gamma}$ is the isotropic part of a covariance function in $\R^d$, for every $d$, provided $\delta\in(0,2]$ and $\lambda>0$. Here, $\gamma$ is a strictly positive scaling parameter.

 \cite{bevilacqua2019estimation} studied
finite and asymptotic properties of the isotropic spectral density of the Cauchy model and  the equivalence of two Gaussian measures with Generalized Cauchy covariance function.

We note that the analytic form of the isotropic spectral densities associated with Mat{\'e}rn and Generalized Cauchy covariances will turn very useful for the developments that follow. Throughout, we use $\widehat{{\cal M}}_{\nu,\alpha}$ and  $\widehat{\cal C}_{\delta,\lambda,\gamma}$, respectively. \\
A well known result about the spectral density of the  Mat{\'e}rn model is the following \citep{Stein1990}:
\begin{equation} \label{stein1}
\widehat{{\cal M}}_{\nu,\alpha}(z)= \frac{\Gamma(\nu+d/2)}{\pi^{d/2} \Gamma(\nu)}
\frac{\alpha^d}{(1+\alpha^2z^2)^{\nu+d/2}}
, \qquad z \ge 0.
\end{equation}
The spectral density associated with the generalized Cauchy function has been studied by \cite{LIM2009}, who proved that
\begin{equation}\label{SDI1}
\widehat{\cal C}_{\delta,\lambda,\gamma}(z)=-\frac{z^{-d}}{2^{\frac{d}{2}-1}\pi^{\frac{d}{2}+1}}
\mathrm{Im}\int_0^\infty \frac{K_{\frac{d}{2}-1}(r) r^{d/2}}{(1+e^{i\pi\delta/2}(r/z\gamma)^\delta)^\lambda}
\,\mathrm{d}r, \qquad z \ge 0.
\end{equation}

\subsection{Contour Loop Integrations and Mellin-Barnes Transforms} \label{Fourier}

We introduce subsequently some technical facts that will turn useful to prove the results contained in Section \ref{sec3}.


Throughout, we denote with $\mathsf{i}$ the unit complex number, and $\langle \cdot,\cdot \rangle $ denotes the inner product in $\R^d$, for $d$ a positive integer.
We start by noting that, in a recent paper,  \cite{Faouzi_2020} provided the following identity\footnote{ We have changed the notation with respect to our previous paper  \cite{Faouzi_2020}. For the quantity on the left hand side of
Eq.(\ref{kondrachuck0}) we  used
$\hat{\cal C}_{d,\gamma}(z;\delta,\lambda\delta)$ in \cite{Faouzi_2020}. However, we have found in the present paper that  $\hat{\cal C}_{\delta,\lambda,\gamma}(z)$ is more compact than $\hat{\cal C}_{d,\gamma}(z;\delta,\lambda\delta)$}
\begin{eqnarray} \label{kondrachuck0}
\hat{\cal C}_{\delta,\lambda,\gamma}(z) &=& \frac{1}{(2\pi)^d}\int_{\R^d} e^{\mathsf{i} \langle \boldsymbol{z}, \boldsymbol{r} \rangle }(1+\norm{\boldsymbol{r}}^\delta\gamma^{-\delta})^{-\lambda}\text{d}^d \boldsymbol{r}, \qquad  \boldsymbol{z} \in \R^d, \; z= \|\boldsymbol{z}\| \nonumber \\
&=& \frac{z^{-d}}{\pi^{d/2}}\frac{1}{\Gamma(\lambda)}\frac{1}{2\pi \mathsf{i}} \int_{-\ve-\mathsf{i}\infty}^{-\ve+\mathsf{i}\infty}   \frac{\Gamma(-u) \Gamma(u +\lambda)\Gamma(d/2+u\delta/2)}
{\Gamma(-u\delta/2)}\left(\frac{2}{z\gamma}\right)^{u\delta} \text{d}u,
\end{eqnarray}
$\ve$ is a small positive real in this paper.

Technical arguments from harmonic analysis \citep[see again][]{Faouzi_2020}  allow to prove that the last identity can be rewritten as

\begin{eqnarray} \label{kondrachuk}
&& \frac{z^{-d}}{\pi^{d/2}}\frac{1}{\Gamma(\lambda)} \Bigg ( \sum_{n=0}^\infty \frac{(-1)^n}{n!}
\frac{\Gamma(\lambda + n)\Gamma(d/2-(\lambda+n)\delta/2)}{\Gamma((\lambda+n)\delta/2)}\left(\frac{z\gamma}{2}\right)^{(\lambda+n)\delta}   \nonumber \\
&+&  \frac{2}{\delta}\sum_{n=0}^\infty \frac{(-1)^n}{n!}
\frac{\Gamma\left(\displaystyle{\frac{2n+d}{\delta}} \right)\Gamma\left(\lambda-  \displaystyle{\frac{2n+d}{\delta}}\right)}{\Gamma(n+d/2)}\left(\frac{z\gamma}{2}\right)^{2n+d} \Bigg ) .
\end{eqnarray}
We define the Fox-Wright function \citep{fox1928asymptotic,wright1935asymptotic} through the identity
\begin{eqnarray*}
{}_p\psi_q \left( \begin{array}{cccc}
(\alpha_1,A_1), & \dots, &   (\alpha_p,A_p) & \\
& & & ;z \\
(\gamma_1,B_1), & \dots, &  (\gamma_q,B_q) &
\end{array}
\right)  = \sum_{k=0}^{\infty}\frac{\prod_{j=1}^{p}\Gamma(\alpha_j + A_jk)}{\prod_{j=1}^{q}\Gamma(\gamma_j + B_jk)}\frac{z^k}{k!}.
\end{eqnarray*}
\cite{Faouzi_2020} prove that (\ref{kondrachuk}) can be written in terms of such a class of functions. Specifically, Equation (\ref{kondrachuk}) is identically equal to
\begin{eqnarray*}
\frac{z^{-d}}{\pi^{d/2}}\frac{1}{\Gamma(\lambda)} \left( \left(\frac{z\gamma}{2}\right)^{\lambda\delta}{}_2\psi_1
\left(
\begin{array}{rrr}
(\lambda,1), & (\frac{d -\lambda\delta}{2},-\frac{\delta}{2}) & \\
& &  ;-\left(\frac{z\gamma}{2}\right)^{\delta}  \\
& (\frac{\lambda\delta}{2}, \frac{\delta}{2}) &
\end{array}
\right)  + \right. \\
\left.
\frac{2}{\delta}\left(\frac{z\gamma}{2}\right)^{d}{}_2\psi_1 \left(
\begin{array}{rrr}
(\frac{d}{\delta},\frac{2}{\delta}), & (\frac{\lambda\delta -d}{\delta},-\frac{2}{\delta}) & \\
& &  ;-\left(\frac{z\gamma}{2}\right)^{2}  \\
& (\frac{d}{2}, 1) &
\end{array}
\right) \right).
\end{eqnarray*}


Another technical argument that is proved helpful for the result coming subsequently is that of Hankel contours. Specifically, we start by rewriting (\ref{kondrachuck0}) as
\begin{equation}
\hat{\cal C}_{\delta,\lambda,\gamma}(z)=  \frac{z^{-d}}{\pi^{d/2}}\frac{1}{2\pi \mathsf{i}} \oint_{\Lambda}\int_0^1
t^{-u-1}(1-t)^{u+\lambda-1}  \frac{\Gamma(d/2+u\delta/2)}{\Gamma(-u\delta/2)}\left(\frac{2}{z\gamma}\right)^{u\delta} \text{d}t  \text{d}u.
\end{equation}
 The contour $\Lambda$ is closed to the left complex infinity and contains the vertical line $(-\ve -\mathsf{i}\infty,-\ve + \mathsf{i} \infty)$, that is customary for the inverse Mellin-Barnes  transformation. \\
We now denote $H$ the Hankel contour, {\em i.e.}, $H=[-i\ve-\infty, -i\ve]\cup~H_{sc}\cup [i\ve,i\ve-\infty]$, with $H_{sc}$ joining $-i\ve$ with $i\ve$ along a positivity oriented semicircle centered at $0$.
Using the definition of Hankel contour in concert with the well known identity
$ {1}/{\Gamma(x)}=({1}/{2i\pi})\oint_{H}e^{s}s^{-x}\text{d}s, $
we get
\begin{equation}
\begin{aligned} & \oint_{\Lambda}\int_0^1
t^{-u-1}(1-t)^{u+\lambda-1}  \frac{\Gamma(d/2+u\delta/2)}{\Gamma(-u\delta/2)}\left(\frac{2}{z\gamma}\right)^{u\delta} \text{d}t  \text{d}u \\
&=\frac{1}{(2\pi~i)}\oint_{\Lambda} \left(\frac{2}{z\gamma}\right)^{u\delta}\Gamma(d/2+u\delta/2)\Bigg ( \oint_{H}~s^{u\delta/2}e^{s} \left ( \int_0^1 ~
t^{-u-1}(1-t)^{u+\lambda-1} \text{d}t \right )   \text{d}s \Bigg ) \text{d}u .
\end{aligned}
\end{equation}
A change  of variable of the type $v=d/2+u\delta/2$ allows to rewrite the triple integral above as 
\begin{equation}
\frac{2 (z\gamma/2)^d}{\delta(2\pi~i)}\oint_{H}~s^{-d/2}e^{s} \Bigg ( \int_0^1
t^{d/\delta-1}(1-t)^{\lambda-d/\delta-1} \Big (\oint_{\tilde{\Lambda}}\Gamma(v) \left[t^{-2/\delta}(1-t)^{2/\delta}s\left(\frac{2}{z\gamma}\right)^2\right]^{v} \text{d}v \Big ) \text{d}t ~ \Bigg ) \text{d}s.
\end{equation}
 Here, $\tilde{\Lambda}$ is a shifted contour  ${\Lambda}$ according to the change of variable from $u$ to $v$ in the complex plane of the inversed Mellin transformation.
Next, we use the known formula $ e^{\frac{-1}{x}}=\frac{1}{2\pi~i}\oint_{\tilde{\Lambda}}\Gamma(v)x^{v}\text{d}v, $ 
in concert with  a change of variable of the type $\tau=\frac{t}{1-t}$ to get that
\def\wg{\widetilde{\gamma}}
\begin{equation}\label{Hankel}
\begin{aligned}
&\oint_{\Lambda} \frac{\Gamma(-u)\Gamma(u+\lambda)}{\Gamma(\lambda)}\frac{\Gamma(d/2+u\delta/2)}{\Gamma(-u\delta/2)}\left(\frac{2}{z\gamma}\right)^{u\delta} \text{d}u\\
&=\frac{z^d \gamma^d}{2^{d-1}\delta}\oint_{H}~s^{-d/2}e^{s} \Big ( \int_0^\infty ~
\frac{\tau^{d/\delta-1}}{(1+\tau)^{\lambda}}e^{-{\tau^{2/\delta}\gamma^2~z^2}/{4s}} \text{d}\tau  \Big ) \text{d}s .\\
\end{aligned}
\end{equation}

\section{Theoretical Results}\label{sec3}
The convergence results below are based on taking a double limit. Some care is needed to take exchangeability into account, as shown below.

\begin{theorem}\label{Converg_unif-1}
Let $\lambda>0$, $\delta\in(d/2,2]$  with $d=1,2.$ Let  $\widehat{\cal M}_{\frac{1}{2},\alpha}$ and $\widehat{\cal C}_{\delta,\lambda,\alpha\lambda}$ be the the isotropic spectral densities defined at (\ref{stein1}) and (\ref{SDI1}), respectively.
Then, it is true that
\begin{equation}\underset{\lambda\to \infty}{\lim} ~   \underset{\delta \to 1}{\lim} ~ \widehat{\cal C}_{\delta,\lambda,\alpha\lambda}(z)=\widehat{\cal M}_{\frac{1}{2},\alpha}(z),
\end{equation}
for $z = \|\boldsymbol{z}\|$ and $\boldsymbol{z} \in \R^d$.
\end{theorem}

\begin{proof}
We provide a constructive proof based on direct calculations.
\begin{eqnarray*}
\underset{\lambda\to \infty}{\lim} ~   \underset{\delta \to 1}{\lim} ~ \hat{\cal C}_{\delta,\lambda,\alpha\lambda}(z) &=& \frac{z^{-d}}{\pi^{d/2}}\frac{1}{2\pi \mathsf{i}} \underset{\lambda\to \infty}{\lim} ~   \underset{\delta \to 1}{\lim}
 \int_{-\ve-\mathsf{i}\infty}^{-\ve+\mathsf{i}\infty}   \frac{\Gamma(-u) \Gamma(u +\lambda)\Gamma(d/2+u\delta/2)}{\Gamma(\lambda)\Gamma(-u\delta/2)}\left(\frac{2}{z\alpha\lambda}\right)^{u\delta} \text{d}u \nonumber\\
&=& \frac{z^{-d}}{\pi^{d/2}}\frac{1}{2\pi \mathsf{i}} \underset{\lambda\to \infty}{\lim} ~
 \int_{-\ve-\mathsf{i}\infty}^{-\ve+\mathsf{i}\infty}   \frac{\Gamma(-u) \Gamma(u +\lambda)\Gamma(d/2+u/2)}{\Gamma(\lambda)\Gamma(-u/2)}\left(\frac{2}{z\alpha\lambda}\right)^u \text{d}u \nonumber\\
&=& \frac{z^{-d}}{\pi^{d/2}}\frac{1}{2\pi \mathsf{i}}  ~
 \int_{-\ve-\mathsf{i}\infty}^{-\ve+\mathsf{i}\infty}   \frac{\Gamma(1/2-u/2)\Gamma(d/2+u/2)}{\pi^{1/2}2^{1+u}}\left(\frac{2}{z\alpha}\right)^u \text{d}u \nonumber\\
 &=& \frac{z^{-d}}{2\pi^{(d+1)/2}}\frac{1}{2\pi \mathsf{i}}  ~
 \int_{-\ve-\mathsf{i}\infty}^{-\ve+\mathsf{i}\infty}   \Gamma(1/2-u/2)\Gamma(d/2+u/2)\left(\frac{1}{z\alpha}\right)^u \text{d}u \nonumber\\
 &=& \frac{z^{-d}}{2\pi^{(d+1)/2}}\frac{1}{2\pi \mathsf{i}}  ~
 \int_{-1-\ve-\mathsf{i}\infty}^{-1-\ve+\mathsf{i}\infty}   \Gamma(-u/2)\Gamma((d+1)/2+u/2)\left(\frac{1}{z\alpha}\right)^{u+1} \text{d}u \nonumber\\
 &=& \frac{z^{-(d+1)}}{\alpha\pi^{(d+1)/2}}\frac{1}{2\pi \mathsf{i}}  ~
 \int_{-1/2-\ve/2-\mathsf{i}\infty}^{-1/2-\ve/2+\mathsf{i}\infty}   \Gamma(-\omega)\Gamma((d+1)/2+\omega)\left(\frac{1}{z\alpha}\right)^{2\omega} \text{d}\omega \nonumber\\
 &=& \frac{z^{-(d+1)}\Gamma((d+1)/2)}{\alpha\pi^{(d+1)/2}}\frac{1}{\le 1+ \frac{1}{z^2\alpha^2}\ri^{(d+1)/2}} \\
 &=& \frac{\alpha^d\Gamma((d+1)/2)}{\pi^{(d+1)/2}}\frac{1}{\le 1+ z^2\alpha^2\ri^{(d+1)/2} }.
 \end{eqnarray*}
 The proof is completed.
\end{proof}
The second result of this section is summarized below.

\begin{theorem}\label{Converg_unif}
Let $\lambda>0$, $\delta\in(d/2,2]$  with $d=1,2.$ Let  $\widehat{\cal M}_{\frac{1}{2},\alpha}$ and $\widehat{\cal C}_{\delta,\lambda,\gamma}$ be the the isotropic spectral densities defined at (\ref{stein1}) and (\ref{SDI1}), respectively.

Then, it is true that
\begin{equation}\label{limDini}  \underset{\delta \to 1}{\lim} ~  \underset{\lambda\to \infty}{\lim}  ~ \widehat{\cal C}_{\delta,\lambda,\alpha\lambda}(z)=\widehat{\cal M}_{\frac{1}{2},\alpha}(z),\end{equation}
for $z = \|\boldsymbol{z}\|$ and $\boldsymbol{z} \in \R^d$.
\end{theorem}

Prior to discussing the proof, some comments are in order. Theorem \ref{Converg_unif-1} and \ref{Converg_unif} show a convergence type argument for the spectral density associated with the
Generalized Cauchy covariance function. Further, the two theorems can be used to argue that the we have exchangeability of the double limit that is taken in both theorems. A proof is provided below.

\begin{proof}
A constructive proof is provided. Let $\hat{{\cal C}}_{\delta,\lambda,\gamma}$ be the Generalized Cauchy spectral density as defined at (\ref{kondrachuck0}).
We invoke the  Mellin Barnes integral \cite{Allendes:2012mr,Faouzi_2020} in concert with Equation (\ref{Hankel}) to write $\hat{{\cal C}}_{\delta,\lambda,\gamma}$ as
\begin{equation}
\begin{aligned}
\widehat{\cal C}_{\delta,\lambda,\alpha\lambda }(z)&=\frac{ (\alpha\lambda)^d}{2^{d-1}\delta\pi^{d/2}(2\pi~\mathsf{i})}\oint_{H}~s^{-d/2}e^{s} \Bigg ( \int_0^\infty ~
\frac{\tau^{d/\delta-1}}{(1+\tau)^{\lambda}}\displaystyle\sum_{n=0}^\infty\frac{(-1)^n}{n!}\left(\frac{\tau^{2/\delta}(\alpha\lambda)^2~z^2}{4s}\right)^n \text{d}\tau  \Bigg )  \text{d}s\\
&=\frac{(\alpha\lambda)^d}{2^{d-1}\delta\pi^{d/2}(2\pi~\mathsf{i})}\displaystyle\sum_{n=0}^\infty\frac{(-1)^n}{n!}\left(\frac{\alpha\lambda z}{2}\right)^{2n}\oint_{H}~s^{-d/2-n}e^{s}\Bigg ( \int_0^\infty  ~
\frac{\tau^{d/\delta+2n/\delta-1}}{(1+\tau)^{\lambda}} \text{d}\tau \Bigg ) \text{d}s.\\
\end{aligned}
\end{equation}
We now note that, for $\lambda$ sufficiently large, we have
$$\int_0^\infty~\frac{\tau^{d/\delta+2n/\delta-1}}{(1+\tau)^{\lambda}} \text{d}\tau = \frac{\Gamma(\lambda-d/\delta-2n/\delta)\Gamma(d/\delta+2n/\delta)}{\Gamma(\lambda)}.$$
Thus, we write
\begin{equation}
\begin{aligned}
\widehat{\cal C}_{\delta,\lambda,\alpha\lambda}(z)&\sim\frac{ (\alpha\lambda)^d}{2^{d-1}\delta\pi^{d/2}(2\pi~\mathsf{i})}\displaystyle\sum_{n=0}^\infty\frac{(-1)^n}{n!}
\left(\frac{\alpha\lambda z}{2}\right)^{2n}\oint_{H}~s^{-d/2-n}e^{s}\frac{\Gamma(\lambda-d/\delta-2n/\delta)\Gamma(d/\delta+2n/\delta)}{\Gamma(\lambda)} \text{d}s\\
&=\frac{\le\alpha\lambda^{1-1/\delta}\ri^d}{2^{d-1}\delta\pi^{d/2}}\displaystyle\sum_{n=0}^\infty\frac{(-1)^n}{n!}\left(\frac{\alpha\lambda^{1-1/\delta}z}{2}\right)^{2n}\frac{\lambda^{d/\delta+2n/\delta}
\Gamma(\lambda-d/\delta-2n/\delta)\Gamma(d/\delta+2n/\delta)}{\Gamma(\lambda)\Gamma(d/2+n)},\\
\end{aligned}
\end{equation}
Direct inspection shows that $$ \underset{\lambda\to\infty}{\lim}\frac{\Gamma(\lambda-d/\delta-2n/\delta)\lambda^{d/\delta+2n/\delta}}{\Gamma(\lambda)}=1.$$ Hence,
\begin{equation}
\begin{aligned}
\widehat{\cal C}_{\delta,\lambda,\alpha\lambda}(z)&\sim\frac{ \le\alpha\lambda^{1-1/\delta}\ri^d}{2^{d-1}\delta\pi^{d/2}}\displaystyle\sum_{n=0}^\infty\frac{(-1)^n}{n!}\frac{\Gamma(d/\delta+2n/\delta)}{\Gamma(n+d/2)}
\left(\frac{\alpha\lambda^{1-1/\delta} z}{2}\right)^{2n}.\\
\end{aligned}
\end{equation}
{We note a classic identity for the convergent series given by $$(1+ax)^{-b}=\displaystyle\sum_{k=0}^{\infty}\frac{(-1)^k(b)_k}{k!}(ax)^k.$$}
The identity above shows the result by letting  $\delta$ tend to 1.
Indeed, $$\underset{\delta\to 1}{\lim}\frac{\Gamma(d/\delta+2n/\delta)}{\Gamma(d/2+n)}=2^{d-1+2n}\Gamma(d/2+n+1/2)/\pi^{1/2}, $$ which in turn implies
$$\underset{\delta\to 1}{\lim}\underset{\lambda\to \infty}{\lim}\widehat{\cal C}_{\delta,\lambda,\alpha\lambda}(z)
= \frac{ \alpha^d}{\pi^{(d+1)/2}}\displaystyle\sum_{n=0}^\infty\frac{(-1)^n}{n!}\Gamma(n+(d+1)/2)\left(\alpha~z\right)^{2n} = \frac{ \alpha^d\Gamma\le d/2 + 1/2 \ri}{\pi^{(d+1)/2}}
\frac{1}{(1 + \alpha^2 z^2)^{(d+1)/2}}.$$
The proof is completed. \end{proof}

We start by recalling the Cauchy uniform criterion of convergence. This will be crucial to prove part of the results following subsequently.

\begin{theorem}\label{UCC}[Uniform Cauchy criterion]
Let $(\mathbb{R},|\cdot|)$ be a complete metric space. Then, a sequence of functions $f_n,$ $n\in\mathbb{N},$  converges uniformly in the sense of Cauchy on the positive real line if and only if
$${\displaystyle \forall \varepsilon> 0 \quad \exists N_{\varepsilon} \in \mathbb{N} \quad \text{s.t.} \quad \forall m, n \in \mathbb{N} \quad \left[m, n \geq N_{\varepsilon} \Rightarrow
\forall x \in \R^+ \quad |f_{m} (x)- f_{n}(x)| \leq \varepsilon \right]}.$$
\end{theorem}

\begin{theorem}\label{Converg_Covariance}
Let $\lambda>0$  and $\delta\in (d/2,2]$. Let
${\cal C}_{\delta,\lambda,\gamma}$ be  the Generalized Cauchy function in Equation (\ref{averrohe}), and let ${\cal M}_{\nu,\alpha}$  be the Mat\'ern function as in Equation (\ref{Matern}). Then, for  $d=1,2$,
$$ \underset{\delta \to 1}{\lim} ~  \underset{\lambda\to \infty}{\lim}        {\cal C}_{\delta,\lambda, \alpha\lambda }(r)= {\cal M}_{\frac{1}{2},\alpha}(r),$$
with uniform convergence  for $r\in (0,\infty)$.
\end{theorem}
\begin{proof}
We need to show that $\widehat{\cal C}_{\delta,\lambda,\alpha\lambda}$ is uniformly convergent  with respect to $\lambda$ and $\delta$.
Let us consider an increasing sequence $\{ \lambda_n\}_{n=0}^{\infty}$ of natural numbers.
 Then, we write
\begin{equation}
\begin{aligned}
\left|\widehat{\cal C}_{\delta, \lambda_m, \alpha\lambda_m}(z)-\widehat{\cal C}_{\delta, \lambda_n, \alpha\lambda_n}(z)\right|& =
\frac{z^{1-d/2}}{\pi^{d/2}}\left|\int_0^\infty \frac{J_{d/2-1}(rz)r^{d/2}}{\left (1+(\frac{r}{\alpha\lambda_m})^\delta \right ) ^{\lambda_m}}\text{d}r
-\int_0^\infty \frac{J_{d/2-1}(rz)r^{d/2}}{\left(1+(\frac{r}{\alpha\lambda_n })^\delta\right)^{\lambda_n}}\text{d}r\right|\\
&\leq \frac{z^{1-d/2}}{\pi^{d/2}}\int_0^\infty \left|\frac{J_{d/2-1}(rz)r^{d/2}}{\left(1+(\frac{r}{\alpha\lambda_m})^\delta\right)^{\lambda_m}}
- \frac{J_{d/2-1}(rz)r^{d/2}}{\left(1+(\frac{r}{\alpha\lambda_n})^\delta\right)^{\lambda_n}}\right|\text{d}r.
\end{aligned}
\end{equation}
Using the well known inequality \citep{MR2723248} $|J_{d/2-1}(rz)|\leq \frac{|rz|^{d/2-1}}{2^{d/2-1}\Gamma(d/2)}$, we have
\begin{equation}
\begin{aligned}
& \left|\widehat{\cal C}_{\delta, \lambda_m, \alpha\lambda_m}(z)-\widehat{\cal C}_{\delta, \lambda_n, \alpha\lambda_n}(z)\right] \\
&\leq \frac{1}{\pi^{d/2}2^{d/2-1}\Gamma(d/2)}\left| \int_0^\infty \left[\frac{r^{d-1}}{\left(1+(\frac{r}{\alpha\lambda_m})^\delta\right)^{\lambda_m}}
- \frac{r^{d-1}}{\left(1+(\frac{r}{\alpha\lambda_n})^\delta\right)^{\lambda_n}}\right]\text{d}r\right|\\
& = \frac{1}{\pi^{d/2}2^{d/2-1}\Gamma(d/2)} \left| \int_0^\infty \left[\frac{\le\alpha\lambda_m\ri^d r^{d-1}}{\left(1+r^\delta\right)^{\lambda_m}}
-  \frac{\le\alpha\lambda_n\ri^d   r^{d-1}}{\left(1 + r^\delta\right)^{\lambda_n}}\right]\text{d}r \right| \\
& = \frac{1}{\delta\pi^{d/2}2^{d/2-1}\Gamma(d/2)} \left| \int_0^\infty \left[\frac{\le\alpha\lambda_m\ri^d r^{d/\delta-1}}{\left(1+r\right)^{\lambda_m}}
-  \frac{\le\alpha\lambda_n\ri^d   r^{d/\delta-1}}{\left(1 + r\right)^{\lambda_n}}\right]\text{d}r \right| \\
& = \frac{1}{\delta\pi^{d/2}2^{d/2-1}\Gamma(d/2)} \left|\le\alpha\lambda_m\ri^d \text{B}\le d/\delta, \lambda_m -d/\delta \ri  - \le\alpha\lambda_n\ri^d \text{B}\le d/\delta, \lambda_n -d/\delta \ri \right| \\
& = \frac{\Gamma( d/\delta)}{\delta\pi^{d/2}2^{d/2-1}\Gamma(d/2)} \left|\le\alpha\lambda_m\ri^d \frac{\Gamma(\lambda_m -d/\delta)}{\Gamma(\lambda_m)}  - \le\alpha\lambda_n\ri^d \frac{\Gamma(\lambda_n -d/\delta)}{\Gamma(\lambda_n)}  \right| \\
& \leq \frac{\Gamma( d/\delta)}{\delta\pi^{d/2}2^{d/2-1}\Gamma(d/2)}\left|\le\alpha\lambda_m\ri^d \lambda_m^{-d/\delta} - \le\alpha\lambda_n\ri^d \lambda_n^{-d/\delta} \right|
\end{aligned}
\end{equation}

Then, for $\delta<1$ with a higher  frequency of $\lambda_i$, $i\in\{n,m\}$,  we obtain that
$${\displaystyle \forall \varepsilon> 0 \quad \exists N_{\varepsilon} \in \mathbb{N}
\quad \forall m, n \in \mathbb{N} \quad \left[m, n \geq N_{\varepsilon} \Rightarrow \forall x \in \R^+ \quad |\widehat{\cal C}_{\delta,\lambda_m,\alpha\lambda_m}(z)- \widehat{\cal C}_{\delta,\lambda_n,\alpha\lambda_n }(z)|
\leq \varepsilon \right]}.$$
 Applying Theorem~\ref{UCC}, we have the result given in Equation~ (\ref{limDini}). The proof is completed.
\end{proof}
Some comments are in order. In the absence of theoretical rates of convergence, we show some simple numerical results on the convergence of the
Generalized Cauchy model ${\cal C}_{1,\lambda, \alpha\lambda }(r)$ to the  Mat{\'e}rn model ${\cal M}_{\frac{1}{2},\alpha}(r)$  when  $\lambda \to \infty$.
Specifically, we analyze  the absolute  error (AE)
\begin{equation}\label{abserr}
E_{\lambda}(r):= |{\cal C}_{1,\lambda, \alpha\lambda }(r) -{\cal M}_{\frac{1}{2},\alpha}(r)|,   \quad r\geq0,
\end{equation}
 when increasing $\lambda$ for  given    values of $\alpha$.

According to Theorem  \ref{Converg_Covariance}, Table  \ref{tabxxx}  depicts  the convergence of the Generalized Cauchy model to    the Mat{\'e}rn model
by reporting the maximum   of the absolute error   $E_{\lambda}(r)$, denoted MAE, when increasing $\lambda$
for different values of $\alpha$.
In particular Table  \ref{tabxxx}   shows that
the MAE approaches zero
when increasing  $\lambda$, as expected. In addition, it can be appreciated that the larger the value of $\alpha$, the faster the convergence  of the Generalized Cauchy model to    the  Mat{\'e}rn model.

   We also depict  a graphical representation for a specific value of $\alpha=0.2/3$
and increasing $\lambda$.
 In particular, in Figure  \ref{cont22} (left part) we plot  ${\cal C}_{1,\lambda, \alpha\lambda }$, 
 for $\lambda=1, 2, 5, 20$
 and ${\cal M}_{1/2,\alpha}$.
 The right part of Figure  \ref{cont22} displays the associated values of $E_{\lambda}$.
 From both Figures,  it can be graphically appreciated the convergence
 under this specific setting.

\begin{table}[h!]
\caption{Maximum of    $E_{\lambda}$
as defined in (\ref{abserr})  when increasing $\lambda$ for increasing value of $\alpha$.} \label{tabxxx}
\begin{center}
\scalebox{1}{
\begin{tabular}{|c||c|c|c|c|c|c|c|c|c|}
  \hline
$\lambda$     &     $1$  & 2&    $5$   &   $20$    &  $60$     & $160$   &  $390$   &   $700$    &  $2000$\\  \hline\hline
$\alpha=0.02/3$   &   0.20363 &0.11608& 0.05075& 0.01331& 0.00449& 0.00169& 0.00069 &0.00039 &0.00014\\
$\alpha=0.2/3$ &    0.20363 &0.11608& 0.05075& 0.01331& 0.00449& 0.00169& 0.00069 &0.00039 &0.00014\\
$\alpha=2/3$   &    0.17687 &0.10340& 0.04620& 0.01228& 0.00415& 0.00156& 0.00064& 0.00036 &0.00013\\
$\alpha=20/3$   &0.00886 &0.00462 &0.00190& 0.00048& 0.00016& 0.00006 &0.00002 &0.00001 &0.00000\\

  \hline
\end{tabular}
}
\end{center}
\end{table}

\begin{figure}[h!]
\begin{tabular}{cc}
\includegraphics[width=7.9cm, height=7.5cm]{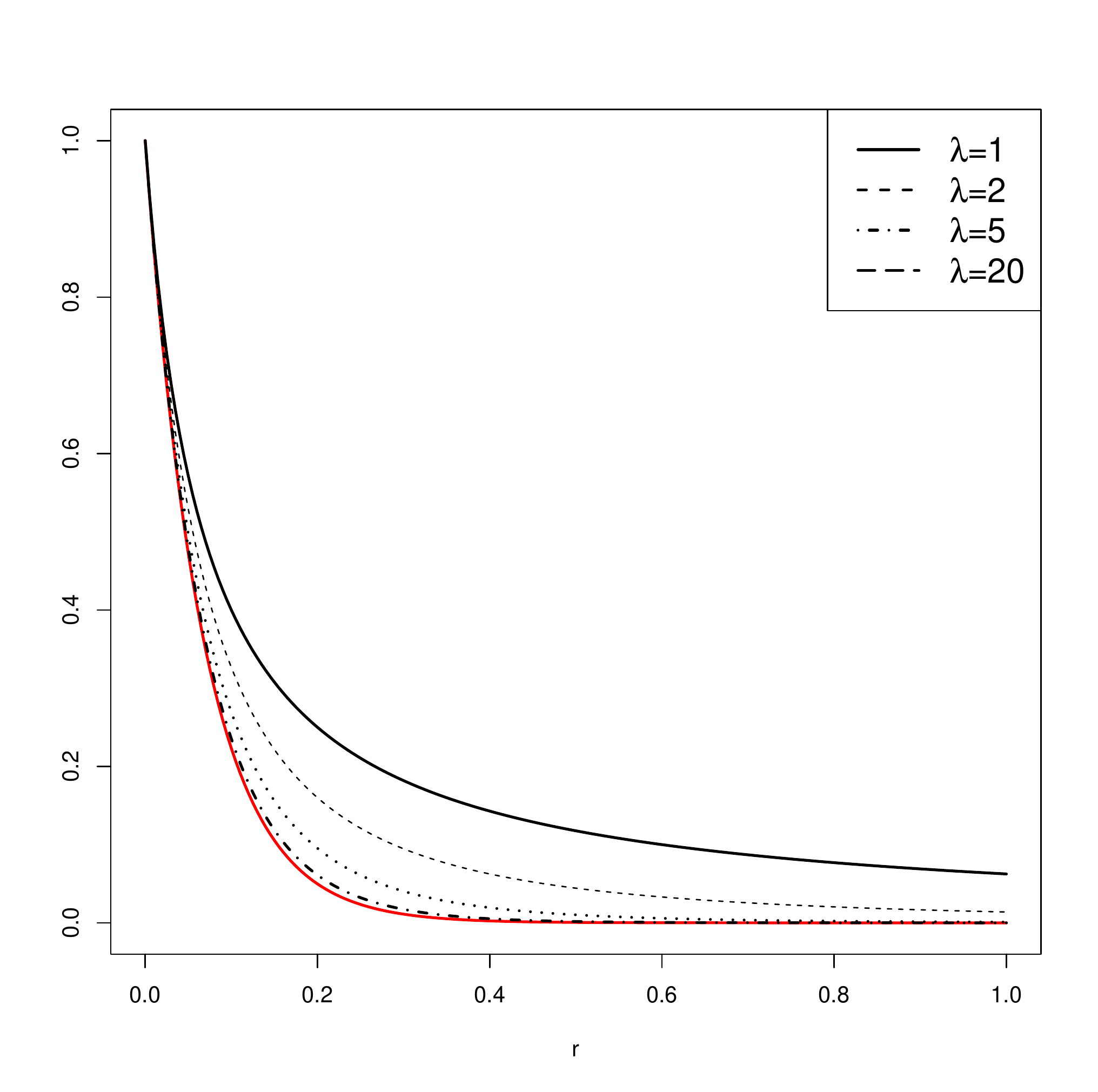} & \includegraphics[width=7.9cm, height=7.5cm]{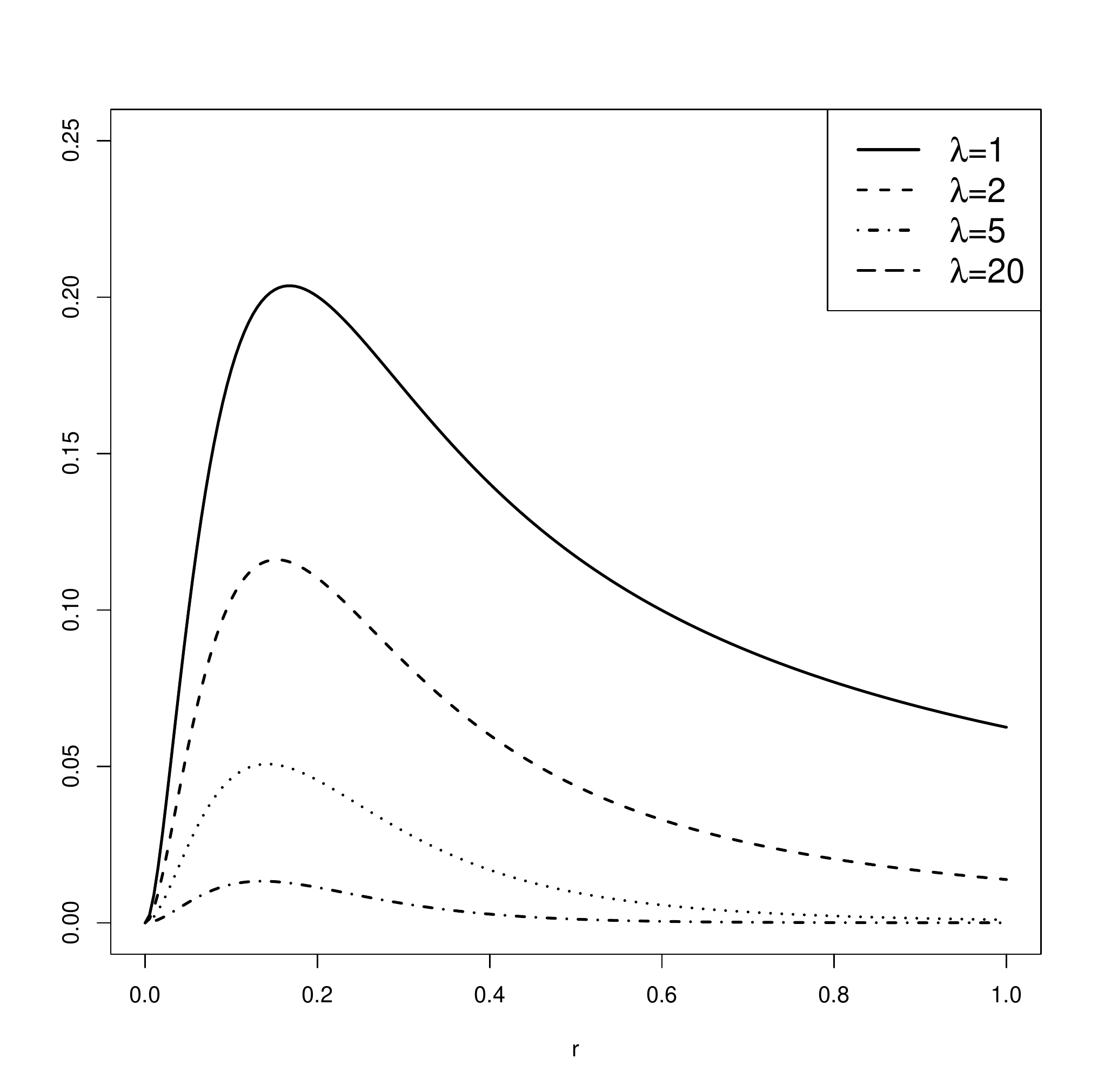} \
\end{tabular}
\caption{Left part: the  ${\cal C}_{1,\lambda, \alpha\lambda }(r)$ model  with $\lambda=1, 2, 5, 20$ and $\lambda \to \infty$ $i.e.$ the
${\cal M}_{\frac{1}{2},\alpha}(r)$ model (red color)
with $\alpha=0.3$.
Right part: associated absolute value error $E_{\lambda}(r)$ as defined in (\ref{abserr})}  \label{cont22}
\end{figure}

\section{Conclusion}\label{sec5}

The present paper provides an interesting bridge between models indexing fractal dimensions and Hurst effects with models that allow to index differentiability at the origin,
and hence mean square differentiability of a Gaussian random field with a given covariance function. We understand that the result obtained this paper applies to a special case
of the Generalized Cauchy function only. It could not be otherwise, as attaining convergence of a long memory process to a process having a covariance function with light tails would be counterintuitive. \\
The techniques used in this paper might open for similar works devoted to bridging different classes of covariance functions, in the spirit of \cite{bevilacqua2022unifying} and \cite{emery2021gauss}. \\
For future researches, it would be extremely useful to bridge different classes of space-time covariance functions.

\section*{Acknowledgement}
Emilio Porcu is grateful to Maaz Musa Shallal for interesting discussions during the preparation of this manuscript.
This paper is based upon work supported by the Khalifa University of Science and Technology under Award No. FSU-2021-016 (E.Porcu).
Partial support was provided  in part by FONDECYT grant 11200749 and in part by grant DIUBB 2020525 IF/R from the university of B\'io B\'io for Tarik Faouzi.
Partial support was provided by FONDECYT grant 1200068 of Chile and by ANID/PIA/ANILLOS ACT210096 for Moreno Bevilacqua.
The work of I.K. was supported in part by Fondecyt (Chile) Grants No. 1121030 and by DIUBB (Chile) Grants No. 2020432 IF/R and  GI 172409/C.

\bibliographystyle{apalike}
\bibliography{fpkb_july_2022}

\end{document}